\documentclass{article}
\usepackage{amsmath}
\usepackage{amsfonts}
\usepackage{amsthm}
\usepackage{amssymb}
\usepackage{color}
\usepackage{enumerate}
\usepackage{tikz}

\newtheorem{theorem}{Theorem}[section]
\newtheorem{lemma}[theorem]{Lemma}

\newtheorem{proposition}[theorem]{Proposition}
\newtheorem{observation}[theorem]{Observation}

\newtheorem{corollary}[theorem]{Corollary}

\theoremstyle{definition}
\newtheorem{definition}[theorem]{Definition}
\theoremstyle{remark}
\newtheorem{remark}[theorem]{Remark}

\newcommand{\ta}{{\rm tail}\hskip0.01cm}
\newcommand{\he}{{\rm head}\hskip0.01cm}
\newcommand{\pr}{{\rm problem}}
\newcommand{\ex}{{\rm excess}}
\newcommand{\rank}{{\rm rank}\hskip0.02cm}
\def\red{\mathbb{R}^E}
\def\rv{\mathbb{R}^V}
\newcommand{\tp}{{\rm TP}\hskip0.02cm}
\newcommand{\tc}{{\rm TC}\hskip0.02cm}

\def\f2{\mathbb{F}_2}

\def\lip{\hskip0.02cm{\rm Lip}\hskip0.01cm}
\def\supp{\hskip0.02cm{\rm supp}\hskip0.01cm}

\newcommand{\ep}{\varepsilon}

\newcommand{\1}{\mathbf{1}}

\begin{document}

\title{Isometric structure of transportation cost spaces on finite metric spaces}

\author{Sofiya Ostrovska and Mikhail I.~Ostrovskii\footnote{Corresponding author, \texttt{ostrovsm@stjohns.edu}}}

\date{\today}
\maketitle

\noindent{\bf Abstract:}  The paper is devoted to isometric
Banach-space-theoretical structure of transportation cost (TC)
spaces on finite metric spaces. The TC spaces are also known as
Arens-Eells, Lipschitz-free, or Wasserstein spaces. A new notion
of a roadmap pertinent to a transportation problem on a finite
metric space has been introduced and used to simplify proofs for
the results on representation of TC spaces as quotients of
$\ell_1$ spaces on the edge set over the cycle space. A
Tolstoi-type theorem for roadmaps is proved, and
 directed subgraphs of the canonical graphs, which are
supports of maximal optimal roadmaps, are characterized. Possible
 obstacles for a TC space on a finite metric space $X$ preventing them from containing
subspaces isometric to $\ell_\infty^n$ have been found in terms of
the canonical graph of $X$. The fact that TC spaces on diamond
graphs do not contain $\ell_\infty^4$ isometrically has been
derived. In addition, a short overview of known results on the
isometric structure of TC spaces on finite metric spaces is
presented.
\medskip

\noindent{\bf Acknowledgement:} The first-named author gratefully
acknowledges the support by Atilim University. This paper was
written while the first-named author was on research leave
supported by Atilim University. The second-named author gratefully
acknowledges the support by the National Science Foundation grant
NSF DMS-1953773.

\begin{large}

\section{Basic definitions and results}

The theory of transportation cost spaces   launched by Kantorovich
and Gavurin \cite{Kan42,KG49} initially had been developed  as a
study of special norms pertaining to function spaces  on finite
metric spaces. However, its further progress turned into the
direction of infinite and continuous setting rather than the
discrete one in papers of Kantorovich and Rubinstein
\cite{Kan42,Kan48,KR57,KR58}, and this stream has become dominant.
See \cite{ABS21, AGS08, FG21, Gar18, KA82, Vil03, Vil09, Wea18}.

Nevertheless, researches within Theoretical Computer Science on
the transportation cost, which computer scientists renamed to
\emph{earth mover distance} \cite{RTG98}, along with  some of the
recent works in metric geometry and Banach space theory focus on
the case of finite metric spaces bringing this area back into
spotlight. See, for example,  \cite{AFGZ21, BMSZ20+, Cha02,
DKO20,DKO21,IT03, KKMR09,KMO20,KN06,Nao21,NR17,NS07,RTG98}. More
details on the history of the subject can be found in
\cite[Chapter 3]{Vil09} and \cite[Section 1.6]{OO19}.

This work deals with the isometric theory of transportation cost
spaces on finite metric spaces. It has to be pointed out that the
transportation cost spaces have been studied from different
perspectives and found numerous applications in various
disciplines. In this connection, there exists a relatively broad
assortment of terms for the same notion. We follow the terminology
going back to Kantorovich up to adding new concepts and names. In
our opinion, this terminology provides the most intuitive
description of the area and, as such, makes it more accessible and
attractive to both mathematicians and researchers working with
practical applications. We presented more arguments in favor of
this selection are in \cite[Section 1.6]{OO19}.

Apart from presenting new results (mentioned in the abstract, see
Sections \ref{S:Digr} and \ref{S:Linfty} for more details), our
goal in this paper is to demonstrate how the convenient
terminology and notation allows to simplify the proofs of some
already available results. For the convenience of the reader, we
present all of the necessary definitions related to the
transportation cost spaces in this article, although most of them
can be found in \cite[Section 1.6]{OO19}.

Let $(X,d)$ be a metric space. Consider a real-valued finitely
supported function $f$ on $X$ with a zero sum, that is,
$\sum_{v\in \supp f}f(v)=0$. A natural and important
interpretation of such a function is considering it as a
\emph{transportation problem}: one needs to transport certain
product from locations where $f(v)>0$ to locations where $f(v)<0$.

One can easily see that a transportation problem $f$ can be
represented as
\begin{equation}\label{E:TranspPlan} f=a_1(\1_{x_1}-\1_{y_1})+a_2(\1_{x_2}-\1_{y_2})+\dots+
a_n(\1_{x_n}-\1_{y_n}),\end{equation} where $a_i\ge 0$,
$x_i,y_i\in X$, and $\1_u(x)$ for $u\in X$ is the {\it indicator
function} of $u$. We call each such representation a
\emph{transportation plan} for $f$, it can be interpreted as a
plan of moving $a_i$ units of the product from $x_i$ to $y_i$. A
pair $x_i,y_i$ for which $a_i>0$ will be called a
\emph{transportation pair}. In some contexts we shorten
``transportation plan'' to \emph{plan} if it does not cause any
confusion.

The \emph{cost} of the transportation plan \eqref{E:TranspPlan} is
defined as $\sum_{i=1}^n a_id(x_i,y_i)$.

We denote the real vector space of all transportation problems by
$\tp(X)$. We introduce the \emph{transportation cost norm} (or
just \emph{transportation cost}) $\|f\|_{\tc}$ of a transportation
problem $f$ as the infimum of costs of transportation plans
satisfying \eqref{E:TranspPlan}.

The completion of $\tp(X)$ with respect to $\|\cdot\|_\tc$ is
called the \emph{transportation cost space} $\tc(X)$. Note that
though for finite metric spaces $X$, the spaces $\tc(X)$ and
$\tp(X)$ coincide as sets, we mostly use the notation $\tc(X)$ to
emphasize that we consider it as a normed space.

It has to be pointed out  that, in our discussion, transportation
plans are allowed to be {\it fake plans}, in the sense that it can
happen that there is no product in $x_i$ in order to make the
delivery to $y_i$. Such fake plans will be used in the proof of
Theorem \ref{T:Tolstoi}.
\medskip

We start with the next statement which is an easy consequence of
the triangle inequality.

\begin{proposition}\label{P:TC_attain} The infimum of costs of transportation plans for $f$ is attained, and this happens for some transportation plan
with $\{x_i\}=\{v:~f(v)>0\}$ and $\{y_i\}=\{v:~f(v)<0\}$.
\end{proposition}

A transportation plan for $f$ of the minimal cost, that is, whose
cost equals $\|f\|_\tc$ is called  an \emph{optimal transportation
plan} for $f\in\tp(X)$.

In the sequel,  a finite metric space $X$ will be identified with
an apposite weighted graph. This weighted graph will be specified
uniquely according to the procedure described below and -
following  \cite{AFGZ21} - will be called the \emph{canonical
graph associated with the metric space}. The graph is defined as
follows. First, we consider a complete weighted graph with
vertices in $X$, and the weight of an edge $uv$ ($u,v\in X$)
defined as $d(u,v)$. After that we delete all edges $uv$ for which
there exists a vertex $w\notin\{u,v\}$ satisfying
$d(u,w)+d(w,v)=d(u,v)$. It is easy to see that, for a finite
metric space, it is a well-defined procedure leading to a uniquely
determined by $X$ weighted graph, which we denote $G(X, E)$, where
$E=E(X)$ is the edge set of the canonical graph.

As a result, the constructed in this way graph $G(X, E)$ possesses
the following  feature, which is crucial for its application in
the theory of the transportation cost spaces:
 the weighted graph distance of
$G(X, E)$ on $X$ coincides with the original metric on $X$. Note
that if $X$ is defined initially as a weighted graph with its
weighted graph distance, then the corresponding canonical graph
can be different because some edges can be dropped. However, if we
start with a simple unweighted graph endowed with the graph
distance, we will recover it as a canonical graph.

For a metric space $X$, let $\ell_{1,d}(E)=\ell_{1,d}(E(X))$ be
the weighted $\ell_1$-space on the edge set of the canonical graph
$G(X,E)$ associated with $X$. The weight of an edge $uv$ is
$d(u,v)$. The space $\ell_{1,d}(E)$ consists of real-valued
functions $\beta:E\to \mathbb{R}$ with the norm
\begin{equation}\label{E:WeiEll1}\|\beta\|_{1,d}:=\sum_{uv\in
E(G)} |\beta_{uv}|d(u,v).\end{equation} It is easy to see that if
$X$ is an unweighted graph with its graph distance, then
$\ell_{1,d}(E)$ is just the vector space $\mathbb{R}^E$ of
functions on the edge set of $X$  with its $\ell_1$-norm.

We introduce the space $\ell_{1,d}(E)$ because, for a weighted
graph, it is natural and practical to consider the following, more
detailed version of transportation plans.

In order to proceed,   an orientation on the edge set $E$ has to
be specified. This  means that we fix a \emph{direction} on the
edges by selecting  the \emph{head} and \emph{tail} of for each
edge. The orientation can be chosen arbitrarily as the notions
important for further reasoning
 do not depend on the choice of this orientation. We call
this orientation a \emph{reference orientation}.

Notation: if $u$ is a tail and $v$ is a head of a directed edge
$e=uv$, we denote the directed edge either $\overrightarrow{uv}$
or $\overleftarrow{vu}$. We also write $u=\ta(e)$ and $v=\he(e)$.

Fix a reference orientation on $G(X,E)$ and consider
$p\in\ell_{1,d}(X)$. We call $p$ a \emph{roadmap} for the reasons
which will be explained below. For each such $p$, we introduce the
function
\begin{equation}\label{E:Problem}
\pr_p(v)=\sum_{\ta(e)=v}p(e)-\sum_{\he(e)=v}p(e)\in\tp(X).
\end{equation}
Note that the function is in $\tp(X)$ because each edge is
included in the right-hand side of \eqref{E:Problem} exactly twice
- once for its head and once for its tail.

It can be readily seen that  $p$ provides an inherent
transportation plan associated with $\pr_p$, namely, the
plan

\begin{equation}\label{E:TPforRDmap}\pr_p=\sum_{e\in\supp
p}p(e)\left(\1_{\ta(e)}-\1_{\he(e)}\right),\end{equation}

\begin{remark}\label{R:CostRdmap} It should be emphasized   that the cost of this
transportation plan is exactly $\|p\|_{1,d}$. This norm is also called
the \emph{cost of roadmap $p$}, because we can identify $p$ and
 transportation plan
\eqref{E:TPforRDmap} by virtue of the canonical correspondence
between those objects.
\end{remark}

\begin{remark}\label{R:Excess} The function $\pr_p(v)$ is a negation of the
well-known in combinatorial optimization (see
\cite[p.~149]{Sch03}) \emph{excess function}: $\ex_p= - \,\pr_p$.
\end{remark}

We use the term \emph{roadmap} because to organize the
transportation (we use intuitive rather than mathematical
language) according to a transportation plan, one has to identify
edges used for transportation and to specify the direction on
them.

Namely, we do the following for the transportation plan given by
the right-hand side of \eqref{E:TranspPlan}, which we denote by
$P$.

\begin{enumerate}[{\rm (1)}]

\item If $x_iy_i$ is not an edge in $E$, we introduce a shortest
path $u_{i,0}=x_i, u_{i,1}, \dots, u_{i,m(i)}=y_i$ joining $x_i$
and $y_i$ in $G$, and replace $a_i(\1_{x_i}-\1_{y_i})$ in plan $P$
by
\[\sum_{j=1}^{m(i)}a_i(\1_{u_{i,j-1}}-\1_{u_{i,j}}).\]

Observe that this step is uniquely determined if and only if for
each transportation pair $x_i,y_i$ of plan $P$ there is one
shortest path between $x_i$ and $y_i$ in $G(X,E)$.

As a result, we get a plan in which every transportation pair is
the pair of ends of an edge, and so it is almost of the form shown
in \eqref{E:TPforRDmap}. We use the word ``almost'' because some
pairs of ends of an edge can repeat in the sum.

\item \label{I:Simpl} We combine all terms corresponding to the
same edge in the obtained transportation plan. We get a
transportation plan for the same problem whose cost does not
exceed the original plan's cost (there can be some cancellations
that will decrease the cost) and for which all transportation
pairs are edges.

\item\label{I:AllEdges} If all transportation pairs are edges, we
map the transportation plan
\[a_1(\1_{x_1}-\1_{y_1})+a_2(\1_{x_2}-\1_{y_2})+\dots+
a_n(\1_{x_n}-\1_{y_n})\] onto the roadmap which on the edge
$x_iy_i$ takes value $a_i$ if the reference orientation of
$x_iy_i$ is $\overrightarrow{x_iy_i}$ and value $-a_i$ if the
reference orientation of $x_iy_i$ is $\overleftarrow{x_iy_i}$. For
edges $uv$ which are not transportation pairs the value of the
roadmap is $0$.

\end{enumerate}

\begin{remark}[Conclusion of the above
construction]\label{R:RDmapForTP} For each transportation plan $P$
for $f\in\tp(X)$, there is a naturally defined (but not uniquely
determined) roadmap $p$ which can be regarded as an implementation
of $P$ and whose cost does not exceed the cost of $P$.
\end{remark}

These definitions and remarks lead to what we consider very
transparent and straightforward proof of the following
proposition.

\begin{proposition}\label{P:QZ} The quotient representation $\tc(X)=\ell_{1,d}(E)/Z$, where $Z$ is the cycle space in $\ell_{1,d}(E)$, holds.
\end{proposition}

As far as we know, Proposition \ref{P:QZ} (in the case of
unweighted graphs) was for the first time proved \cite[Proposition
10.10]{Ost13}. Later, different versions of the proof were
published in \cite{OO20}, \cite{AFGZ21}, \cite{DKO21}. Our version
of the proof can be regarded as a generalization of the proof in
\cite{DKO21} for unweighted graphs to the case of metric spaces.

Prior to presenting this version, let us remind the notion of the
cycle space and other related concepts of the algebraic graph
theory, used  in the subsequent parts of this paper.

The following definition, as a rule, is used for an arbitrary
oriented finite graph $G=(V, E)$, not necessarily connected. We
are going to use it for canonical graphs $G=(X, E)$ with reference
orientation.

Denote by $\red$ and $\rv$ the spaces of real-valued functions on
the edge set and the vertex set, respectively.

\begin{definition}\label{D:CycleSpEtc}
The \emph{incidence matrix} $D$ of an oriented graph $G$ is
defined as a matrix whose rows are labelled using vertices of $G$,
columns are labelled using edges of $G$, and the $ve$-entry is
given by
\[
d_{ve}=\begin{cases} 1, & \hbox{ if }v=\he(e),\\
-1, & \hbox{ if }v=\ta(e),\\
0, & \hbox{ if $v$ is not incident to }e.\end{cases} \]

Interpreting elements of $\red$ and $\rv$ as column vectors, we
may regard $D$ as a matrix of a linear operator $D:\red\to \rv$.
We also consider the transpose matrix $D^T$ and the corresponding
operator $D^T:\rv\to \red$.\medskip

It is easy to describe $\ker D^T$. In fact, for $f\in \rv$ the
value of $D^T(f)\in \red$ at an edge $e$ is $f(\he(e))-f(\ta(e))$,
therefore $f\in\ker D^T$ if and only if it has the same value at
the ends of each edge. It is clear that this happens if and only
if $f$ is constant on each of the connected components of $G$.
Therefore the ranks of the operators $D^T$ and $D$ are equal to
$|V|-c$, where $c$ is the number of connected components of
$G$.\medskip

Observe that
\begin{equation}\label{E:DF}(Dp)(v)=-\,\pr_p(v)\end{equation} for $p\in  \red$. We
let \begin{equation}\label{E:DefZ} Z=\ker D.\end{equation} This
subspace of $\red$ is called the \emph{cycle space}\index{cycle
space} or \emph{cycle subspace}.\index{cycle subspace} The name
was chosen because, as we explain below, the set of signed
indicator functions of all cycles spans this subspace.\medskip

Now we consider a cycle $C$ in $G$ (in a graph-theoretical sense).
We consider one of the two possible orientations of $C$ satisfying
the following condition: each vertex of $C$ is a head of exactly
one edge and a tail of exactly one edge. Now we introduce the
\emph{signed indicator function} $\chi_C\in \red$ of the oriented
cycle\index{signed indicator function of an oriented cycle} $C$
(in an oriented graph $G$) by
\begin{equation}\label{E:SignInd}
\chi_C(e)=\begin{cases} 1 & \hbox{ if }e\in C\hbox{ and its
orientations in $C$ and $G$ are the same}\\
-1 & \hbox{ if }e\in C\hbox{ but its orientations in $C$ and $G$
are different}
\\
0 & \hbox{ if }e\notin C.
\end{cases}\end{equation}

It is immediate from \eqref{E:TPforRDmap} and \eqref{E:DF} that
$\chi_C\in\ker D$. Let us show that signed indicator functions of
cycles span $\ker D$. To see this we observe that $\dim(\ker
D)=|E(G)|-\rank(D)=|E(G)|-|V(G)|+c$. So it remains to show that
there are $|E(G)|-|V(G)|+c$ \,cycles $\{C(i)\}_i$ such that the
functions $\{\chi_{C(i)}\}_i$ are linearly independent.\medskip
\medskip

To achieve this goal, consider a maximal subset of edges of $G$
which does not form any cycles. Such collection of edges of $G$
induces a subgraph $H$ of $G$ which is called a \emph{spanning
forest of $G$}; in the case where $G$ is connected, it is called a
\emph{spanning tree of $G$}. Notice that we use these terms in a
slightly different way than it can be found in the literature.
This discrepancy stems from different possible treatments of
 the graph's
isolated vertices.  Since we  do not need to consider graphs with
isolated vertices, this issue is of no importance here. It is a
well-known fact that the number of edges in a spanning forest is
$|V(G)|-c$. Therefore, the number of the remaining edges is
$|E(G)|-|V(G)|+c$. Now we construct the cycles $C(i)$, $i\in
E(G)\backslash E(H)$, with linearly independent $\{\chi_{C(i)}\}$
as follows. Each $i\in E(G)\backslash E(H)$ forms a cycle together
with some edges of $H$. One can show that such a cycle is uniquely
determined, but at this point, it does not matter. We pick one
such cycle and denote it $C(i)$. The fact that $\{\chi_{C(i)}\}$
are linearly independent is immediate because the value of the
linear combination $\sum_i a_i\chi_{C(i)}$ at an edge $i$ is equal
to $\pm a_i$.\end{definition}

\begin{proof}[Proof of Proposition \ref{P:QZ}] Our idea of the proof is very
straightforward: to  apply the made above observation that using
the reference orientation, any element $p$ of $\ell_{1,d}(E)$ can
be regarded as a roadmap for some transportation problem, namely,
for $f=\pr_p$. In such a way, one gets a natural linear map
$Q:\ell_{1,d}(E)\to \tc(X)$.

Remark \ref{R:RDmapForTP} shows that this map is surjective and
that it has norm $\le 1$. Finally, the fact that, for each
$f\in\tc(X)$, there is an optimal transportation plan $P$ with
cost $\|f\|_\tc$, implies that the roadmap $p$ assigned to $P$ -
see Remark \ref{R:RDmapForTP} -  satisfies $\|p\|_{1,d}=\|f\|_\tc$
(the inequality $\|p\|_{1,d}<\|f\|_\tc$ would contradict the
definition of an optimal transportation plan). Thus
$Q:\ell_{1,d}(E)\to (\tc(X), \|\cdot\|_\tc)$ is a quotient map.
Finally, combining \eqref{E:DF} with the definition \eqref{E:DefZ}
of $Z$, we conclude that $Z=\ker Q$.\end{proof}

\begin{remark} There exist analogues of Proposition \ref{P:QZ}
for some infinite metric spaces, see \cite{CKK17,GL18}.
\end{remark}

Proposition \ref{P:QZ} implies that one can improve a non-optimal
roadmap for a given transportation problem by adding an element of
the cycle space. Our next target is to prove a ramification of
this fact, stating that the roadmap can be improved by adding a
multiple of a cycle. This is exactly the assertion of Theorem
\ref{T:Tolstoi}  in Section \ref{S:Tolstoi}.
\medskip

The idea of this result goes back to Tolstoi \cite{Tol30,Tol39},
see \cite[Theorem 12.1, Theorem 12.3, Theorem 21.12, and
pp.~362--366]{Sch03} for interesting related information. Tolstoi
\cite{Tol30,Tol39} was the first one who proposed describing the
optimality conditions for plans via nonexistence of ``negative''
cycles. The listed above theorems in \cite{Sch03} describe such
conditions of optimality in different settings.

\subsection{An overview of the isometric theory of $\tc(X)$ for
finite $X$}\label{S:Survey}

Since the isometric theory of transportation cost spaces on finite
metric spaces is quite recent, it is beneficial not only for the
presentation of this study but also for further developments in
this area to precede the exposition of new findings with a short
survey of the available ones.

\begin{enumerate}[{\bf (1)}]

\item \label{I:Godard} The space $\tc(X)$ is isometric to
$\ell_1^n$ if and only if $X$ is isometric to a weighted tree on
$n+1$ vertices. Godard \cite{God10} proved the ``if'' part. In the
``only if'' direction, the paper \cite{God10} contains only a
partial result that $X$ embeds into a tree. The ``only if'' result
can be easily derived from the description of extreme points in
the unit ball of $\tc(X)$ as vectors of the form
$(\1_u-\1_v)/d(u,v)$, where $u$ and $v$ such that all triangle
inequalities of the form
\[d(u,z)+d(z,v)\ge d(u,v)\]
are strict for $z\notin\{u,v\}$. The result is obtained by
comparing the number of extreme points in $\ell_1^{n+1}$ and
$\tc(X)$, see \cite[Proposition 26]{DKO21} for details. It should
be remarked here  that the study of extreme points in $\tc(X)$ was
apparently initiated by Weaver in the first edition of his book
\cite{Wea99}, and the mentioned above elementary description of
extreme points in the unit ball of $\tc(X)$ for finite $X$ is a
part of folklore in this study.

\item Isometric description of $\tc(X)$ as a quotient of
$\ell_1(E)$ over the cycle space was obtained in \cite{Ost13} in
the case when  $X$ is an unweighted graph. This result was
generalized to the case of weighted graphs in \cite{OO20}.
Alternative proofs of it were obtained in \cite{AFGZ21},
\cite{DKO21}, and this paper.

\item The study of isometric embeddability of $\ell_1$ into
$\tc(X)$ for infinite $X$ was initiated by C\'uth and Johanis
\cite{CJ17}. However, the following seems to be the first
isometric result in the case of finite $X$:

\begin{theorem}[{{\rm Khan, Mim, and Ostrovskii }\cite{KMO20}}]\label{T:2npts} If a metric space $M$ contains $2n$ elements, then $\tc(M)$ contains a
$1$-complemented subspace isometric to $\ell_1^n$. If the space
$M$ is such that triangle inequalities for all distinct triples in
$M$ are strict, then $\tc(M)$ does not contain a subspace
isometric to $\ell_1^{n+1}$.
\end{theorem}

\item The paper \cite{KMO20} also contains examples of finite
metric spaces $X_3$ and $X_4$ such that $\tc(X_3)$ contains
$\ell_\infty^3$ isometrically and $\tc(X_4)$ contains
$\ell_\infty^4$ isometrically. The examples were simplified in
\cite{DKO21} where it was shown that one can take $X_3=K_{2,4}$
(complete bipartite graph) and $X_4=K_{4,4}$. Finally, in
\cite{AFGZ21} it was observed that $\tc(C_4)=\ell_\infty^3$.

Note that the problem on the existence of isometric copies of
$\ell_\infty^n$ in $\tc(X)$ can be regarded as an isometric
version of the famous Bourgain's problem (see \cite{Bou86} and the
discussion in \cite[Section 1.2.2]{Nao18}) on the cotype of
$\tc(\mathbb{R}^2)$.

\item Alexander, Fradelizi, Garc\'ia-Lirola, and Zvavitch
\cite{AFGZ21} characterized isometries of $\tc(X)$ in terms of the
canonical graph $G(X)$. They showed that isometries correspond to
cycle-preserving bijections of edge sets. This allowed them to use
the results of Whitney \cite{Whi33} (see also
\cite[Section~5.3]{Oxl11}) on classifications of cycle-preserving
bijections of edge sets. In particular, they proved that, for a
$3$-connected graph, isometries correspond to homotheties of the
graph - that is, the multiplication of all weights of edges by the
same positive number. Earlier, a similar result was obtained for
Sobolev spaces on graphs, see \cite{Ost07}.

\item The result of \cite{AFGZ21} on $\ell_1$-decompositions:
Decompositions of $\tc(X)$ of the form $\tc(X)=Z_1\oplus_1 Z_2$
correspond to representability of the canonical graph $X$ as a
union of two graphs $X_1$ and $X_2$ with one common vertex, such
that $Z_1=\tc(X_1)$ and $Z_2=\tc(X_2)$.

\item The  result of \cite{AFGZ21} on
$\ell_\infty$-decompositions: Decompositions of $\tc(X)$ of the
form $\tc(X)=Z_1\oplus_\infty Z_2$ imply that one of the summands
of $Z_1$  and $Z_2$ is one-dimensional, while the other is
$\ell_1^{n-1}$ and $X=K_{2,n}$ (unweighted complete bipartite
graph).

\item Findings of \cite{AFGZ21} reveal that transportation cost
spaces are significant examples for the studies related to the
Mahler conjecture on the volume product of a symmetric convex body
and its polar.
\end{enumerate}

\subsection{A new Tolstoi-type theorem}\label{S:Tolstoi}

 The setting of our Tolstoi-type theorem is different from the ones
in \cite{Sch03} and the proof method for  the ``only if''
direction is different as well. For this reason, we decided that
its publication will be useful for further studies of
transportation cost spaces. Meanwhile, the proof of  the ``if''
direction is easily established in all cases.

For a roadmap $p$, we introduce the \emph{induced by $p$ direction
(orientation) of edges which are in $\supp p$} as follows: If
$p(e)>0$, the orientation of the edge coincides with the reference
orientation; if $p(e)<0$ - it is the opposite to the reference
orientation.

Here comes our Tolstoi-type theorem.

\begin{theorem}\label{T:Tolstoi} Let $p$ be a roadmap for some transportation problem
$f$ on a weighted graph $G=(X,E)$. The roadmap $p$ is not optimal
if and only if the graph $G$ contains a directed cycle such that
the total weight of the edges in it which are in $\supp p$ and
whose induced by $p$ direction is opposite to their direction in
the cycle exceeds the total weight of other edges in the directed
cycle.
\end{theorem}

\begin{proof}[Proof of Theorem \ref{T:Tolstoi}] The ``if'' direction. Assume that there is such an oriented
cycle $c$. We identify it with $\chi_c$ defined in
\eqref{E:SignInd} and consider it as an element of
$\ell_{1,d}(E)$. The assumption on $c$ implies that, for
sufficiently small $\alpha>0$, one has $\|p+\alpha
c\|_{1,d}<\|p\|_{1,d}$. On the other hand, since $c$ is an
oriented cycle, it is clear that $\pr_{p+\alpha c}=\pr_p=f$. Thus,
 the roadmap $p$ is not optimal.
\medskip

To prove the ``only if'' part of Theorem \ref{T:Tolstoi}, assume
that there are cheaper than $p$ roadmaps for $f$. The set of such
roadmaps includes optimal roadmaps for $f$. By compactness, among
the optimal roadmaps, there is a roadmap $\tilde p$ which
minimizes $\|\tilde p-p\|_{1,d}$. If such $\tilde p$ is not
unique, we pick one of them.

As $d:=\tilde p-p$  is a difference of two roadmaps for the same
transportation problem, it is a roadmap for the null problem, that
is, the transportation problem in which nothing is available and
nothing is needed. Using the orientation induced by $d,$ we create
an oriented graph $(X,E_d)$ with the same vertex set as $G$ and
with edge set being $\supp d$.

Since $\pr_d=0$, it is clear that each vertex with nontrivial
indegree has a nontrivial outdegree in $(X,E_d)$. Hence, this
oriented graph has an oriented cycle $c$. We are going to prove
that this cycle satisfies the condition in Theorem
\ref{T:Tolstoi}.

Assume the contrary. Then the following condition for $c$ holds:

(*) The total weight of edges of $c$, whose directions are
opposite to the induced by $p$ does not exceed the total weight of
the remaining edges in $c$.

We identify $c$ with $\chi_c$ and consider it as a roadmap.  The
choice of $c$ implies that for each $\delta\in(0,\min_{e\in
c}d(e))$, there holds $\|d-\delta c\|_{1,d}<\|d\|_{1,d}$, or,
equivalently,

\begin{equation}\label{E:Decr} \|(\tilde p-\delta c)-p\|_{1,d}<\|\tilde p-p\|_{1,d}.
\end{equation}

Selecting arbitrary $\delta\in(0,\min_{e\in c}d(e))$, set $\hat
p=\hat p_\delta:=\tilde p-\delta c$. Now, we refine the choice of
$\delta$ by showing that there exists $\delta \in(0,\min_{e\in
c}d(e))$ for which $\|\hat p\|_\tc\le \|\tilde p\|_\tc$. Indeed,
the definition of $c$ implies that the signs of $c$ and $\tilde
p$, considered as vectors of $\ell_{1,d}(E)$, can differ only on
edges where the sign of $c$ coincides with the sign of $-p$. By
condition (*), the total length of all such edges in $c$ does not
exceed half of length of $c$. Bearing this in mind, opt for
$\delta \in(0,\min_{e\in c}d(e))$ which does not exceed the
minimum absolute value of $\tilde p(e)$ over all of those edges
$e$ on which the signs of $\tilde p$ and $c$ coincide. Then, for
this $\delta,$ one has:
 $\|\hat p\|_{1,d}=\|\tilde p-\delta c\|_{1,d}\le \|\tilde
p\|_{1,\delta}$.

Furthermore, for any such $\delta$, roadmap
 $\hat p$ is a roadmap for the original transportation
problem as we subtracted a roadmap for a null problem, and, by
virtue of \eqref{E:Decr}, one has:
 $\|\hat p-p\|_{1,d}<\|\tilde p-p\|_{1,d}$.

Thus, $\hat p$ is also an optimal roadmap for $f$, and it
satisfies  $\|\hat p-p\|_{1,d}<\|\tilde p-p\|_{1,d}$, contrary to
our choice of $\tilde p$.\end{proof}

\section{Directed graphs related to transportation problems and Lipschitz
functions}\label{S:Digr}

The dual space of $\tc(X)$ is the space $\lip_0(X)$ of Lipschitz
functions on $X$ which vanish at a base point (an arbitrarily
chosen point in $X$), see \cite[Section 10.2]{Ost13}. Let
$f\in\tc(X)$ and $l\in\lip_0(X)$ be such that $\lip(l)=1$ and
$l(f)=\|f\|_\tc$, where by $l(f)$ we mean the result of the action
of a functional $l$ on a  vector $f$. If
$f=a_1(\1_{x_1}-\1_{y_1})+a_2(\1_{x_2}-\1_{y_2})+\dots+
a_n(\1_{x_n}-\1_{y_n}),\quad a_i>0,$ is an optimal transportation
plan for $f$,  then \[\sum_{i=1}^n
a_i(l(x_i)-l(y_i))=\|f\|_\tc=\sum_{i=1}^n a_i d(x_i,y_i),\] whence
\begin{equation}\label{E:Potential} l(x_i)-l(y_i)=d(x_i,y_i)\end{equation} for each transportation pair
$x_iy_i$.

A transportation plan satisfying condition \eqref{E:Potential} for
some $1$-Lipschitz function $l$ is called \emph{potential} with
\emph{potential} $l$.

\begin{observation}[Kantorovich and Gavurin \cite{KG49}]\label{O:KG49}
A transportation plan is optimal if and only if it is potential.
\end{observation}

According to the preceding discussion, we consider optimal (=
minimal cost) roadmaps for $f\in\tc(X)$ as elements of
$\ell_{1,d}(E)$ and define the \emph{support} of a roadmap as the
support of the corresponding element of $\ell_{1,d}(E)$. Let us
show that there always exists an optimal roadmap $p$ for $f$ with
the largest possible support among the roadmaps for $f$, whence
the largest possible support of a roadmap is uniquely determined
by $f$. In fact, any two optimal for $f$ roadmaps, $p_1$ and
$p_2$, cannot be represented by elements of $\ell_{1,d}(E)$ having
different signs on any of the edges, because otherwise, the
element $\frac12(p_1+p_2)\in \ell_{1,d}(E)$ would be also a
roadmap for $f$ with a strictly smaller cost. Thus, considering in
$\ell_{1,d}(E)$ an average of a collection of optimal roadmaps
covering all edges belonging to the support of at least one of the
optimal roadmaps, one obtains an optimal roadmap with the maximal
possible support.

For an optimal roadmap $p$ with the maximal possible support,
consider the following directed graph, which we call the
\emph{directed graph of $f$}. It contains only those edges of $E$
that are in the support of $p$. Further, the edges with positive
value of $p$   have the reference orientation, while, for edges
where the value of $p$ is negative, the orientation is the
opposite-to-reference one. To put differently, the directions of
all edges coincide with the directions of transportation according
to the roadmap $p$.

Our next purpose  is to provide a convenient description for the
obtained class of directed subgraphs of $(X,E)$. This will be done
in terms of the definition below.

\begin{definition}\label{D:DownGr} Let $l$ be a $1$-Lipschitz function on $G=(X,E)$.
Consider the set of edges $uv\in E$ for which
$|l(u)-l(v)|=d(u,v)$, such an edge is called {\it supported by}
$l$. A \emph{downhill graph} for $l$ is  the directed graph
consisting of all supported by $l$ edges directed from the vertex
where $l$ is larger to the vertex where $l$ is smaller.
\end{definition}

By Observation \ref{O:KG49} and the remark  on roadmaps with the
maximal possible support preceding Definition \ref{D:DownGr}, the
directed graph of a transportation problem is a subgraph of a
downhill graph for some $1$-Lipschitz function.

\begin{theorem}\label{T:DGTPvsDH} A directed subgraph of $(X,E)$ with at least one edge is a downhill graph for a $1$-Lipschitz function
if and only if it is the directed
graph for some non-zero transportation problem.
\end{theorem}

Prior to proving the Theorem, some auxiliary results will be
presented. The following definition comes in handy.

\begin{definition} Given  $0\neq f\in\tp(X)$, a {\it supporting function}  $s$  for $f$ is a function  $s\in\lip_0(X)$ 
 satisfying $\lip (s)=1$ and $s(f)=\|f\|_\tc$. \end{definition}

Denote by $T_f$ the subset of $E(X)$ consisting of all edges of
$E(X)$ which are contained in supports of some optimal roadmaps
for $f$. Trivially, $(X,T_f)$ is a \emph{spanning subgraph} of
$(X, E(X))$, that is, a subgraph containing all vertices of the
graph.

A necessary and sufficient condition for the uniqueness of a
supporting function is given by the forthcoming statement.

\begin{theorem}\label{T:Unique} A supporting function for $f$ is unique if and only if the graph $(X,T_f)$ is connected.
If the graph $(X,T_f)$ is disconnected, then a supporting function
for $f$ is uniquely determined on the connected component of
$(X,T_f)$ containing $O$, and is determined up to additive
constants on other connected components of $(X,T_f)$.
\end{theorem}

\begin{proof} Let $p$
be an optimal roadmap for $f$ whose support is $T_f$:
\[p=\sum_{j=1}^m a_j\1_{\overrightarrow{x_jy_j}},\quad a_j>0,\]
and $T_f=\{x_jy_j\}_{j=1}^m$. Then,
\[s(f)=\sum_{j=1}^ma_j(s(x_j)-s(y_j)).\]
On the other hand, the above sum is equal to
$\|f\|_\tc=\sum_{j=1}^m a_jd(x_j,y_j)$. Since $\lip (s)=1$, this
occurs if and only if $s(x_j)-s(y_j)=d(x_j,y_j)$ for each
$x_jy_j\in T_f$ yielding that $s$ is determined up to an additive
constant on each of the components of $(X,T_f)$. Taking into
account that $s(O)=0$,  one concludes that $s$ is uniquely
determined on the component of $(X,T_f)$ containing $O$.
\medskip

Our proof of the ``only if'' part of Theorem \ref{T:Unique} is
based on the following two lemmas.

\begin{lemma}\label{L:InOnEComp} Downhill edges for $s$, whose both ends belong to the same component
of $T_f$,    are edges of $T_f$.
\end{lemma}

\begin{proof} Assume that  $uv$ is an edge with $s(u)-s(v)=d(u,v)$ satisfying the conditions of the lemma.
Let $u=u_0,u_1,\dots,u_k=v$ be a path in $T_f$ joining $u$ and
$v$. Then, for each $i=1,\dots,k$, one has: $s(u_i)-s(u_{i-1})=\pm
d(u_i,u_{i-1})$. Now, consider the cycle $u_0,u_1,\dots,u_k,u_0$
together with
 a maximal-support roadmap $p$ for $f$. By the assumptions above,
the  support of $p$ contains all edges of this cycle except
$u_ku_0$. Denote by $\ep$ ($\ep>0$)  the minimal amount of product
transported according to the roadmap $p$ along the edges
$u_0u_1,u_1u_2,\dots,u_{k-1}u_k$ in the direction of decrease of
$s$.

Next, we create, for the same problem, a new roadmap $\tilde p$
from $p$ as follows. In the roadmap $\tilde p$, we move $\ep$
units of product along the edge $u_0u_k$ in the direction of
decrease of $s$. To modify $p$ only for edges in the cycle
$u_0,u_1,\dots,u_k,u_0$, we proceed as follows: selecting the
direction on the cycle which on $u_0u_k$ coincides with the
direction of decrease of $s$, we follow this orientation. From
here on,  for each edge on which the direction of the cycle and of
the decrease of $s$ coincide, we increase the amount of
transported product by $\ep$, while on each edge for which the
direction of decrease of $s$ and the direction of the cycle are
opposite, the amount of transportation  will be decreased by
$\ep$. It is easy to see that
\[\pr_p=\pr_{\tilde p}\quad\hbox{and}\quad \|p\|_{1,d}=\|\tilde
p\|_{1,d}.\] Therefore $\tilde p$ is a roadmap for the same
problem and $uv$ is in its support.
\end{proof}

\begin{lemma}\label{L:BetwComp} For each edge $e=uv$ joining
different components of $(X,T_f)$, there exists a supporting
function $s_e$ for $f$ such that $|s_e(u)-s_e(v)|\ne d(u,v)$.
\end{lemma}

\begin{proof} Only edges which are downhill for $s$ have to be considered.
Let $X_1,\dots,X_k$ be components of $(X,T_f)$. If there is $u\in
X_i$ and $v\in X_j$, $j\ne i$, such that $s(v)-s(u)=d(u,v)$, we
write $X_i\prec X_j$.

We start by proving that there are no cycles in this ordering of
components, that is, there are no finite subsets
\[\{n(1),\dots,n(m)\}\subset\{1,\dots,k\},  \quad\mathrm{where}\quad m\ge 2,\] such that
\begin{equation}\label{E:OrdCycle}X_{n(1)}\prec X_{n(2)}\prec\dots\prec X_{n(m)}\prec
X_{n(1)}.\end{equation}

Assume the contrary, that is, that there exists such a cycle. Pick
vertices $u_i\in X_{n(i)}$ and $v_i\in X_{n(i+1)}$ - bearing in
mind that $v_m\in X_{n(1)}$ - in such a way that
$s(v_i)-s(u_i)=d(u_i,v_i)$.

Denote by $W_i$  a path in $T_f$ joining $v_{i-1}$ and $u_i$,
where, by the agreement, $W_1$ is a path in $T_f$ joining $v_m$
and $u_1$.  Then,
\begin{equation}\label{W} u_1v_1W_2u_2v_2W_3u_3\dots v_{m-1}W_mu_mv_mW_1u_1\end{equation} is a
cycle. Observe that, with respect to the function $s$, all of the
edges $u_iv_i$ are going uphill. Let $p$ be the roadmap for $f$,
whose support is $T_f$, and let $\ep>0$ be the minimal amount of
the product moved according to $p$ over all edges of
$W_1,W_2,\dots, W_m$. Let us add to $p$ a plan of moving $\ep$
unit of products along the cycle \eqref{W}, in the direction which
contains edges $v_iu_i$. Notice, that such a plan is feasible
since all of the edges $v_iu_i$ have the same direction in the
cycle. This change of the plan will increase the resulting cost
 by $\ep\times$(the total length of all downhill
edges traversed in the downhill direction),  and simultaneously
will reduce this cost  by $\ep\times$(the total length of all
downhill edges traversed in the uphill direction). Our choice of
$\ep$ guarantees that the statement of the previous sentence
holds. Therefore, the cost of the newly constructed roadmap $\hat
p$ equals the cost of $p$, and hence $\hat p$  will be an optimal
roadmap for the same problem $f$. Consequently, all $u_iv_i$ are
edges in $T_f$, yielding a contradiction.
\medskip

Now, let $\overrightarrow{uv}$ be a downhill edge for $s$ for
which we are going  to prove Lemma \ref{L:BetwComp}. Let $X_1$ and
$X_2$ be the components of $(X,T_f)$ joined by $uv$; $u\in X_1$
and $v\in X_2$.

To construct the desirable function $s_e$, we introduce the
following definition. We say that a component $X_i$ is
\emph{reachable down} from $X_2$ if there is a set $\{n(1),\dots,
n(m)\}\subset \{1,\dots,k\}$ such that
\[X_i\prec X_{n(1)}\prec X_{n(2)}\prec\dots\prec X_{n(m)}\prec X_{2}.\]

We split all components of $(X,T_f)$ into the set $U$ of
components which are reachable down from $X_2$ and the set $V$ of
components which are not reachable down from $X_2$; we include
$X_2$ in $U$. By the nonexistence of cycles of the form
\eqref{E:OrdCycle}, $X_1$ is among those components which are not
reachable down from $X_2$, whence $X_1$ is in $V$.

The definition of being reachable down implies that there is a
$\delta>0$ such that, for each edge with end $w$ in $V$ and end
$z$ in $U$, one has $s(z)-s(w)<d(z,w)-\delta$.

In the case when $O\in V$, put
\begin{equation}\label{E:Modif_s}s_e(x)=\begin{cases} s(x) &\hbox{ if }x\in V,\\
s(x)+\frac{\delta}2&\hbox{ if }x\in U,\end{cases}\end{equation}
and when $O\in U$, put
\begin{equation}\label{E:Modif_s2}s_e(x)=\begin{cases} s(x)-\frac{\delta}2 &\hbox{ if }x\in V,\\
s(x)&\hbox{ if }x\in U.\end{cases}\end{equation}

In either case, $s_e$ is also a supporting function for $f$, while
for $s_e$ there are no downhill edges from $U$ to $V$ and from $V$
to $U$. This completes the proof.
\end{proof}

To derive Theorem \ref{T:Unique} from Lemma \ref{L:BetwComp},
consider the case of disconnected $(X,T_f)$, pick a supporting
function $s$ for $f$, and observe that we need to examine the two
cases: (1) There exists an edge $e$ between two distinct
components of $(X,T_f)$ which is downhill for $s$; (2) There are
no such edges.

Lemma \ref{L:BetwComp} proves Theorem \ref{T:Unique} in Case (1),
because $s_e$ is obviously different from $s$.

Case (2). Since all considered graphs are finite, there exists
$\omega>0$ such that, for any edge $uv$ between different
components of $(X, T_f)$, $|s(u)-s(v)|\le d(u,v)-\omega$. Let
$X_0$ be the component of $(X,T_f)$ containing $O$. Then the
function $\tilde s$ given by
\begin{equation}\label{E:Modif_s3}\tilde s(x)=\begin{cases} s(x) &\hbox{ if }x\in X_0,\\
s(x)+\frac{\delta}2&\hbox{ otherwise},\end{cases}\end{equation}

is another supporting function for $f$. This proves Theorem
\ref{T:Unique}.
\end{proof}

After establishing  all the necessary grounding, let us come back
to the proof of Theorem   \ref{T:DGTPvsDH}.

\begin{proof}[Proof of Theorem \ref{T:DGTPvsDH}] (a) Proof of ``only if'' part. Let $(X,E_l)$ be the downhill graph for a $1$-Lipschitz function $l$.
For each edge $\overrightarrow{uv}\in E_l$ consider the
transportation problem $\1_u-\1_v$. We add all such problems over
all $\overrightarrow{uv}\in E_l$ and get a transportation problem
\begin{equation}\label{E:TrPforEl} f=\sum_{\overrightarrow{uv}\in
E_l}\left(\1_u-\1_v\right),\end{equation} for which the right-hand
side of \eqref{E:TrPforEl} is an optimal transportation plan by
Observation \ref{O:KG49}.

This plan has the largest support among optimal transportation
plans for $f$, because, on one hand, $l(f)=\|f\|_\tc$, and, on the
other hand $|l(\1_x-\1_y)|<d(x,y)$ for any edge $xy$ which is not
in the plan \eqref{E:TrPforEl}.
\medskip

(b)  Proof of ``if'' part. Taking average of support functions
$s_e$ constructed in Lemma \ref{L:BetwComp} over all  edges $e$
between different components of $T_f$, we obtain a support
function $s_a$, all of whose downhill edges are within one of the
components of $(X,T_f)$. By Lemma \ref{L:InOnEComp}, $T_f$ is the
downhill graph of $s_a$.
\end{proof}

\section{Nonexistence of isometric copies of $\ell_\infty^k$ in $\tc(X)$ for
some finite graphs $X$}\label{S:Linfty}

By the Maurey-Pisier theorem \cite{MP76}, the problem of presence
of isometric copies of $\ell_\infty^k$ in transportation cost
spaces is closely related to the problem of determining the cotype
of transportation cost spaces, which is a significant open problem
for many metric spaces. See the survey \cite[Section
1.2.2]{Nao18}, where this problem is restated in terms of
universality.

Since $\ell_\infty^2$ is isometric to $\ell_1^2$, it is easy to
check that $\ell_\infty^2$ is isometrically contained in $\tc(X)$
for every $X$ with at least $4$ points. Alternatively, it can be
derived as  a consequence of Theorem \ref{T:2npts}. On the other
hand, if $X$ has $3$ points, it is isometric to $\ell_\infty^2$ if
and only if $X$ is a tree. See the result \eqref{I:Godard} in
Section \ref{S:Survey}.

If $X$ contains $C_4$ with distortion $1$, then, by virtue of the
observation in \cite{AFGZ21} saying that $\tc(C_4)=\ell_\infty^3$,
$\tc(X)$ contains $\ell_\infty^3$ isometrically.

To probe more in this direction, there are two known examples of
finite spaces containing $\ell_\infty^4$ isometrically.
Historically the first one  is a discrete subset of the unit
sphere of $\ell_\infty^4$ in its $\ell_\infty$ metric discovered
in \cite{KMO20}. The second one is the complete bipartite graph
$K_{4,4}$, see \cite{DKO21}.

This section purports to detect obstacles preventing an isometric
containment of $\ell_\infty^k$ in $\tc(X)$. To do so, we suppose
that $\ell_\infty^k$ admits an isometric embedding into $\tc(X)$
for a finite metric space $X$. Let $(X,E(X))$ be the canonical
graph of $X$ and let a sequence $\{e_i\}_{i=1}^k\subset \tc(X)$ be
isometrically equivalent to the unit vector basis of
$\ell_\infty^k$.

In the rest of this section, the elements $\{e_i\}$ will be
presented by their optimal roadmaps with an understanding that
such a presentation is not unique. As before, a reference
orientation of $E(X)$ is taken to be fixed.

To elaborate more on the number of optimal roadmaps for elements
$e_i,$ we introduce the following definition.

\begin{definition}\label{D:StrDisj} Two transportation problems are called \emph{strongly
disjoint} if the maximal supports of their optimal roadmaps in
$\ell_{1,d}(E)$ are disjoint. Equivalently, two transportation
problems are  \emph{strongly disjoint} if any two roadmaps for
them are disjoint as vectors in $\ell_{1,d}$.
\end{definition}

It is not difficult to see that any two transportation problems
$f,g\in \tc(X)$, where $X$ is a finite weighted graph, which are
isometrically equivalent to the unit vector basis of
$\{\ell_1^2\}$ are strongly disjoint. Indeed, if this does not
hold,  then either $\|f+g\|_\tc<\|f\|_\tc+\|g\|_\tc$ or
$\|f-g\|_\tc<\|f\|_\tc+\|g\|_\tc$.\medskip

This remark leads to the following statement.

\begin{observation}\label{O:CUSign} Any pair $f_j,~j=1,2$ of transportation problems of the form
$\sum_{i=1}^ka_{i,j}e_i$, for which  $|a_{i,j}|\le 1$ and also
$1=a_{l,1}a_{l,2}=-a_{k,1}a_{k,2}$ for some pair of indices $k,l$,
is strongly disjoint.
\end{observation}

\begin{proof} Since $\{e_i\}_{i=1}^k$ is isometrically equivalent to the unit vector
basis of $\ell_\infty^k$,  problems $\{f_j\}_{j=1}^2$ are
isometrically equivalent to the unit vector basis of $\ell_1^2$,
the statement holds by the remark above.
\end{proof}

The following result on the number of optimal roadmaps for elements $\{e_i\}$ holds.

\begin{proposition}\label{P:ManyDisRM}
 For every $e_j$, there are at least $2^{k-2}$ disjoint optimal
roadmaps in $\ell_{1,d}(E(X))$.
\end{proposition}

\begin{proof} Given real numbers $\{a_i\}_{i=1}^k$,  set
\begin{equation}\label{E:2^{n-2}}
f(a_1,a_2,\dots,a_k)=\sum_{i=2}^ka_ie_i.
\end{equation}

For each collection $\theta_1,\theta_2,\theta_3,\dots,\theta_k$,
$\theta_i=\pm 1$, we pick some optimal roadmap
$p(\theta_1,\dots,\theta_k)$ for $f(\theta_1,\dots,\theta_k)$.
\medskip

By Observation \ref{O:CUSign}, problems
\[f(1,\theta_2,\dots,\theta_{j-1},1,\theta_{j+1},\dots,\theta_k)~\hbox{
and }~
f(1,\theta_2,\dots,\theta_{j-1},-1,\theta_{j+1},\dots,\theta_k)\]
are strongly disjoint. Their sum equals
$2f(1,\theta_2,\dots,\theta_{j-1},0,\theta_{j+1},\dots,\theta_k)$,
while their difference equals  $2e_j$. This implies that
\begin{equation}\label{E:RMe_j}
\frac12\left(p(1,\theta_2,\dots,\theta_{j-1},1,\theta_{j+1},\dots,\theta_k)-
p(1,\theta_2,\dots,\theta_{j-1},-1,\theta_{j+1},\dots,\theta_k)\right)\end{equation}
is an optimal roadmap for $e_j$ and
\begin{equation}\label{E:RMrest}
\frac12\left(p(1,\theta_2,\dots,\theta_{j-1},1,\theta_{j+1},\dots,\theta_k)+
p(1,\theta_2,\dots,\theta_{j-1},-1,\theta_{j+1},\dots,\theta_k)\right)\end{equation}
is an optimal roadmap for
$f(1,\theta_2,\dots,\theta_{j-1},0,\theta_{j+1},\dots,\theta_k)$
Strong disjointness of
$f(1,\theta_2,\dots,\theta_{j-1},1,\theta_{j+1},\dots,\theta_k)$
and
$f(1,\theta_2,\dots,\theta_{j-1},-1,\theta_{j+1},\dots,\theta_k)$
implies that the roadmaps \eqref{E:RMe_j} and \eqref{E:RMrest}
have the same support.

By Observation \ref{O:CUSign}, problems
\[f(1,\theta_2,\dots,\theta_{j-1},0,\theta_{j+1},\dots,\theta_k)~\hbox{
and
}~f(1,\tilde\theta_2,\dots,\tilde\theta_{j-1},0,\tilde\theta_{j+1},\dots,\tilde\theta_k)\]
are strongly disjoint for any two distinct $(k-2)$-tuples
$(\theta_2,\dots,\theta_{j-1},\theta_{j+1},\dots,\theta_k)$ and
$(\tilde\theta_2,\dots,\tilde\theta_{j-1},\tilde\theta_{j+1},\dots,\tilde\theta_k)$
consisting of $\pm1$. Therefore roadmaps \eqref{E:RMe_j} are
disjoint for distinct $(k-2)$-tuples
$(\theta_2,\dots,\theta_{j-1},\theta_{j+1},\dots,\theta_k)$.\end{proof}

Employing Proposition \ref{P:ManyDisRM} one arrives at:

\begin{theorem}\label{T:DimBound} If $\tc(X)$ contains an isometric copy of
$\ell_\infty^k$, then the canonical graph $G(X,E)$ of $X$ contains
vertices whose degrees are at least $2^{k-2}$. Furthermore, the
number of such vertices should be nontrivially large, namely, if
$\{e_i\}_{i=1}^k\subset \tc(X)$ are isometrically equivalent to
the unit vector basis of $\ell_\infty^k$, then each vertex
contained in a support of any of $\{e_i\}_{i=1}^k$ should have
degree at least $2^{k-2}$. \end{theorem}

Note that $\{e_i\}_{i=1}^k$ do not have to be disjointly
supported, see \cite[Section 8]{DKO21}.
\medskip

It is worthy to state the following immediate consequence of
Theorem \ref{T:DimBound}  because it is related to the Bourgain's
problem on cotype of $\tc(\mathbb{R}^2)$ which we mentioned in
Section \ref{S:Survey}, see also \cite{NS07} in this connection.

\begin{corollary}\label{C:Grid} The space $\tc(L_n)$, where $L_n$ is the plane
$n\times n$ grid, does not contain $\ell_\infty^5$ isometrically.
\end{corollary}

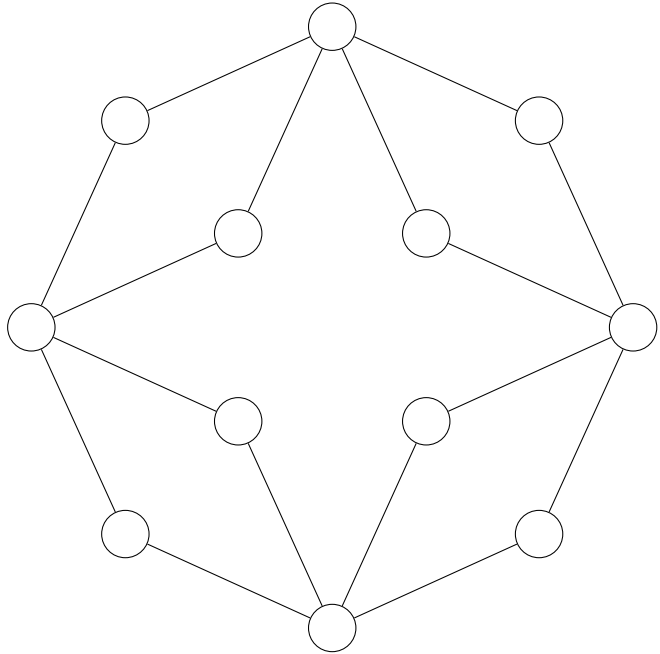
\begin{figure}
\begin{center}
{\begin{tikzpicture}
 [scale=.25,auto=left,every node/.style={circle,draw}]
\node (n1) at (16,0) {\hbox{~~~}};
 \node (n2) at (5,5)  {\hbox{~~~}};
 \node (n3) at (11,11)  {\hbox{~~~}};
 \node (n4) at (0,16) {\hbox{~~~}};
 \node (n5) at (5,27)  {\hbox{~~~}};
 \node (n6) at (11,21)  {\hbox{~~~}};
 \node (n7) at (16,32) {\hbox{~~~}};
 \node (n8) at (21,21)  {\hbox{~~~}};
 \node (n9) at (27,27)  {\hbox{~~~}};
\node (n10) at (32,16) {\hbox{~~~}};
 \node (n11) at (21,11)  {\hbox{~~~}};
  \node (n12) at (27,5)  {\hbox{~~~}};

  \foreach \from/\to in {n1/n2,n1/n3,n2/n4,n3/n4,n4/n5,n4/n6,n6/n7,n5/n7,n7/n8,n7/n9,n8/n10,n9/n10,n10/n11,n10/n12,n11/n1,n12/n1}
  \draw (\from) -- (\to);

\end{tikzpicture}} \caption{Diamond $D_2$.}\label{F:Diamond2}
\end{center}
\end{figure}

To formulate  the next corollary, let us remind the definition of
diamond graphs.

\begin{definition}\label{D:Diamonds} Diamond graphs $\{D_n\}_{n=0}^\infty$
are defined recursi\-ve\-ly: The {\it diamond graph} of level $0$
has two vertices joined by an edge of length $1$ and is denoted by
$D_0$. The {\it diamond graph} $D_n$ is obtained from $D_{n-1}$ in
the following way. Given an edge $uv\in E(D_{n-1})$, it is
replaced by a quadrilateral $u, a, v, b$, with edges $ua$, $av$,
$vb$, $bu$. (See Figure \ref{F:Diamond2}.)
\end{definition}

Apparently Definition \ref{D:Diamonds} was first introduced in
\cite{GNRS04}. We consider $D_n$ as a weighted graph - the weight
of each edge is $2^{-1}$. It is clear that with these weights the
vertex set of $D_{n-1}$ is naturally isometrically embeddable into
the vertex set of $D_n$. We call those vertices of $D_n$ which are
in the set $V(D_n)\backslash V(D_{n-1})$ \emph{vertices of the
$n$-the generation}. Observe that all vertices of the $n$-th
generation in $D_n$ have degree $2$.

\begin{corollary}\label{C:Diamond} The spaces $\{\tc(D_n)\}_{n=1}^\infty$ do not contain an isometric copy of
$\ell_\infty^4$.
\end{corollary}

\begin{proof} Suppose that $\tc(D_n)$ contains $\ell_\infty^4$ isometrically, and $\{e_j\}_{j=1}^4\subset\tc(D_n)$ are isometrically equivalent to the unit
vector basis of $\ell_\infty^4$. By Proposition \ref{P:ManyDisRM},
each vertex belonging to support of one of $\{e_j\}_{j=1}^4$ has
degree at least $4=2^{4-2}$. As observed in the paragraph above,
none of these vertices is of the $n$-th generation, and all of
them are in $V(D_{n-1})$. Therefore, $\tc(D_{n-1})$ also contains
$\ell_\infty^4$ isometrically. Eventually,  we arrive at a
contradiction because the maximal degree of vertices of $D_1$
equals  $2$.
\end{proof}

\begin{remark} Corollary \ref{C:Diamond} is sharp: by the observation of \cite{AFGZ21} mentioned at the beginning of this section, $\tc(D_n)$ contains an isometric copy of $\ell_\infty^3$
for every $n\in \mathbb{N}$.
\end{remark}

It is easy to see that a similar argument can be used for
recursive sequences of graphs introduced by Lee and Raghavendra
\cite{LR10}:

\begin{definition}\label{D:Comp}
Let $H$ and $G$ be two finite connected directed graphs having
distinguished vertices which we call {\it top} and {\it bottom},
respectively. The {\it composition} $H\oslash G$ is obtained by
replacing each  edge $\overrightarrow{uv}\in E(H)$ by a copy of
$G$, the vertex $u$ is identified with the bottom of $G$ and the
vertex $v$ is identified with the top of $G$. Directions of edges
in $H\oslash G$ are inherited from $G$. The {\it top} and {\it
bottom} of the obtained graph are defined as the top and bottom of
$H$, respectively.
\end{definition}

When we consider these graphs as metric spaces we use the graph
distances of the underlying undirected weighted graphs - that is,
we ignore the directions of edges.

Defining metrics on such graphs we make the following
normalization. We assume that the distance between the top and
bottom of $G$ is equal to $1$. When we replace an edge $e\in E(H)$
by a copy of $G$, we multiply all weights of edges in this copy of
$G$ by $w(e)$. This normalization is chosen because under this
normalization the natural embedding of $V(H)$ into $V(H\oslash G)$
is isometric.
\medskip

Let $B$ be a connected weighted finite simple directed graph
having two distinguished vertices, which we call {\it top} and
{\it bottom}, respectively. Assume that the distance between the
top and bottom in $B$ is $1$. We use $B$ to construct a recursive
family of graphs as follows:

\begin{definition}\label{D:B_n} We say that  graphs $\{B_n\}_{n=0}^\infty$ are defined by {\it recursive
composition} or that $\{B_n\}_{n=0}^\infty$ is a {\it recursive
sequence} or {\it recursive family} of graphs if:

\begin{itemize}

\item A  graph $B_0$ consists of one directed edge of length $1$
with {\it bottom} being the initial vertex and {\it top} being the
terminal vertex.

\item $B_n=B_{n-1}\oslash B$.

\end{itemize}
\end{definition}

The  weights of edges in $B_n$ are defined  as described above.
Now, one can formulate another  corollary of Theorem
\ref{T:DimBound}:

\begin{corollary} Let $\Delta$ be the maximum degree of $B$. Then
the spaces $\{\tc(B_n)\}_{n=0}^\infty$ do not contain
$\ell_\infty^k$ isometrically for $k>\log_2\Delta+2$.
\end{corollary}

\begin{proof} Suppose $\tc(B_n)$ contains $\ell_\infty^k$ isometrically, and $\{e_j\}_{j=1}^k\subset\tc(B_n)$ are isometrically equivalent to the unit
vector basis of $\ell_\infty^k$. By Proposition \ref{P:ManyDisRM},
each vertex belonging to support of one of $\{e_j\}_{j=1}^k$ has
degree at least $2^{k-2}>2^{(\log_2\Delta+2)-2}=\Delta$. By the
definition of $B_n$, none of these vertices can belong to any
graph $B$ replacing an edge of $B_{n-1}$ in the last step of
construction of $B_n$, except the top or bottom of $B$ (degree can
be $>\Delta$ only at the top and bottom of $B$). Thus all vertices
in the support of $\{e_j\}_{j=1}^k$ belong to the vertex set of
$B_{n-1}$. Therefore $\tc(B_{n-1})$ also contains $\ell_\infty^k$
isometrically, and we can repeat the argument for $B_{n-1}$.
Eventually, we arrive at a contradiction because the maximal
degree of vertices of $B_1=B$ equals  $\Delta$.
\end{proof}

\begin{small}

\renewcommand{\refname}{
\end{small}

\textsc{Department of Mathematics, Atilim University, 06830
Incek,\\ Ankara, TURKEY} \par \textit{E-mail address}:
\texttt{sofia.ostrovska@atilim.edu.tr}\par\medskip

\textsc{Department of Mathematics and Computer Science, St. John's
University, 8000 Utopia Parkway, Queens, NY 11439, USA} \par
  \textit{E-mail address}: \texttt{ostrovsm@stjohns.edu} \par

\end{large}

\end{document}